\documentclass[a4paper]{amsart}
\usepackage[latin1]{inputenc}
\usepackage{amssymb}
\usepackage{amsmath}
\usepackage{mathrsfs}
\usepackage{eufrak}
\usepackage{amsthm}
\usepackage{amsfonts}
\usepackage{textcomp}
\usepackage{graphicx}
\usepackage[pdftex]{color}
\usepackage{paralist}
\usepackage[shortlabels]{enumitem}
\usepackage{hyperref}
\usepackage{comment}
\usepackage[arrow, matrix, curve]{xy}

\newtheorem*{corollary*}{Corollary}
\newtheorem{theorem}{Theorem}[section]
\newtheorem*{theoremp*}{Theorem}

\newtheorem{corollary}[theorem]{Corollary}
\newtheorem{lemma}[theorem]{Lemma}
\newtheorem{proposition}[theorem]{Proposition}

\newtheorem*{claim*}{Claim}

\theoremstyle{definition}
\newtheorem{definition}[theorem]{Definition}

\newtheorem{example}[theorem]{Example}

\theoremstyle{remark}

\numberwithin{equation}{theorem}

\makeatletter
\renewcommand*\env@matrix[1][\
arraystretch]{%
  \edef\arraystretch{#1}%
  \hskip -\arraycolsep
  \let\@ifnextchar\new@ifnextchar
  \array{*\c@MaxMatrixCols c}}
\makeatother

\begin{document}

\title{Upper bounds for dominant dimensions of gendo-symmetric algebras}
\date{\today}



\author{Ren\'{e} Marczinzik}
\address{Institute of algebra and number theory, University of Stuttgart, Pfaffenwaldring 57, 70569 Stuttgart, Germany}
\email{marczire@mathematik.uni-stuttgart.de}

\begin{abstract}
The famous Nakayama conjecture states that the dominant dimension of a non-selfinjective finite dimensional algebra is finite.
In \cite{Yam}, Yamagata stated the stronger conjecture that the dominant dimension of a non-selfinjective finite dimensional algebra is bounded by a function depending on the number of simple modules of that algebra. With a view towards those conjectures, new bounds on dominant dimensions seem desirable.
We give a new approach to bounds on the dominant dimension of gendo-symmetric algebras via counting non-isomorphic indecomposable summands of rigid modules in the module category of those algebras. On the other hand, by Mueller's theorem, the calculation of dominant dimensions is directly related to the calculation of certain Ext-groups.
Motivated by this connection, we generalize a theorem of Tachikawa about non-vanishing of $Ext^{1}(M,M)$ for a non-projective module $M$ in group algebras of $p$-groups to local Hopf algebras and we also give new results for showing the non-vanishing of $Ext^{1}(M,M)$ for certain modules in other local selfinjective algebras, which specializes to show that blocks of category $\mathcal{O}$ and 1-quasi-hereditary algebras with a special duality have dominant dimension exactly 2. In the final section we raise different questions with the hope of a new developement on those conjectures in the future.
\end{abstract}

\maketitle
\section*{Introduction}
We always assume that we work with finite dimensional, connected and non-semisimple algebras over a field $K$ and all modules are finite dimensional right modules, if nothing is stated otherwise.
Recall that the famous Nakayama conjecture (see \cite{Nak}) states that the dominant dimension of any non-selfinjective algebra is finite. A special case of this conjecture is the Tachikawa conjecture, which states that $Ext^{i}(M,M) \neq 0$ for some $i \geq 1$ for any non-projective module in a selfinjective algebra.
The connection between those conjectures is, for example, explained in \cite{Yam}. The most important class of selfinjective algebras are the symmetric algebras, which includes for example all group algebras. In this special situation the truth of the Tachikawa conjecture for symmetric algebras is equivalent to the truth of the Nakayama conjecture for gendo-symmetric algebras, where gendo-symmetric algebras were first defined in \cite{FanKoe} as algebras isomorphic to endomorphism rings of generators over symmetric algebras. Motivated by this we seek to find upper bounds of the dominant dimension and develop new tools to attack the Nakayama conjecture for gendo-symmetric algebras in section 2, while section 1 gives the preliminaries.
For an $A$-module $M$ we call the number of nonisomorphic indecomposable direct summands the size of $M$ and denote it by size($M$). For a subcategory $C$ of mod-$A$ we say $C \perp_r C$ iff $Ext^{k}(X,Y)=0$, for all $X,Y \in C$ and $k=1,2,...,r$.
We define $o_k(A):=sup \{ $ size(M) $|$ $M \in mod-A$, add($M$) $\perp_k$ add($M$) $ \}$.
Our first main result is as follows:
\begin{theoremp*}
(see \ref{maintheorem})
Let $A$ be a non-selfinjective gendo-symmetric algebra and let $w$ denote the number of simple $A$-modules.
Then $(o_k(A)+2-w)(k+2)-1 \geq domdim(A)$, for all $k \geq 1$.
\end{theoremp*}
As a corollary of this theorem, we obtain a generalisation of the known fact that the Nakayama conjecture is true for gendo-symmetric algebras with representation-dimension at most 3 and some other special cases. We also note that at the moment there seems to be no known example of a gendo-symmetric algebra $A$ with $o_k(A)$ infinite for all $k \geq 1$ and thus the previous theorem might be seen as a new approach to prove Tachikawa's conjecture for symmetric algebras. In the rest of section 2, we give similar bounds for the  here newly introduced class of 1-Extsymmetric algebras, which generalize the class of weakly 2-Calabi-Yau algebras. Note that the class of weakly 2-Calabi-Yau algebras contains the class of preprojective algebras of Dynkin type. See for example section 5 in \cite{CheKoe} for more background on this.
Another motivation for our previous theorem is a conjecture by Yamagata, which is stronger than the Nakayama conjecture. In \cite{Yam}, Yamagata conjectured that the dominant dimension of any non-selfinjective algebra is bounded by a finite function depending on the number of simple modules of that algebra. Our inequality in the previous theorem suggests to attack this conjecture also by calculating $o_k(A)$ for certain $k$ and hoping that one obtains something, which depends only on the number of simple modules of $A$. We show that this really works if in addition we assume that $A$ is a Nakayama algebra.
By an old result of Mueller (see \cite{Mue}) the dominant dimension of an algebra $A$ with dominant dimension at least two and minimal faithful projective injective left $A$-module $Af$ (for some idempotent $f$) can be calculated by calculating $\inf \{ i \geq 1 | Ext_{fAf}^{i}(fA,fA) \neq 0 \}+1$ inside $mod-fAf$. This motivates to study the non-vanishing of certain Ext-groups in the module category of a given algebra, which we do in section 3. Tachikawa succeded to show that $Ext^{1}(M,M) \neq 0$ for any non-projective module in a group algebra of a $p$-group. Note that every group algebra of a $p$-group is a local Hopf algebra. We generalise Tachikawa's result to a general local Hopf algebra and also provide an example of a local Hopf algebra, which is not isomorphic to a group algebra of a p-group to show that our result is really a generalisation. We state the theorem here in a slightly different but equivalent form compared to the main text: 
\begin{theoremp*}
(see \ref{tachtheo})
Let $A$ be a local Hopf algebra and $M$ an arbitrary non-projective $A$-module, then $Ext^{1}(M,M) \neq 0$.
\end{theoremp*}
Further results in section 3 include that $Ext_A^{1}(M,M) \neq 0$ for twosided ideals $M$ over a local symmetric algebra $A$, which gives as a corollary that 1-quasi-hereditary algebras with a special duality (in the sense of theorem B in \cite{Pu}) and blocks of category $\mathcal{O}$ have dominant dimension exactly two. The last section asks several questions motivated by the statements in this work.
The author is thankful to Xingting Wang for suggesting the example in \ref{hopfexample} and Steffen Koenig for helpful comments.

\section{Preliminaries}
In this article we always assume that $A$ is a finite dimensional and connected algebra over a field $K$. To avoid trivialities, we assume that $A$ is not semisimple. We always work with finite dimensional right modules, if not stated otherwise. $mod-A$ denotes the category of finite dimensional right $A$-modules.
$D:=Hom_K(-,K)$ denotes the $K$-duality of an algebra $A$ over the field $K$.
For background on representation theory of finite dimensional algebras and their homological algebra, we refer to \cite{ASS}, \cite{SkoYam} and \cite{ARS}.
For a module $M$, $add(M)$ denotes the full subcategory of $mod-A$ consisting of direct summands of $M^n$ for some $n \geq 1$.
A module $M$ is called basic in case $M \cong M_1 \oplus M_2 \oplus ... \oplus M_n$, where every $M_i$ is indecomposable and $M_i$ is not isomorphic to $M_j$ for $i \neq j$. The basic version of a module $N$ is the unique (up to isomorphim) module $M$ such that $add(M)=add(N)$ and such that $M$ is basic.
We denote by $S_i=e_iA/e_iJ$, $P_i=e_i A$ and $I_i=D(Ae_i)$  the simple, indecomposable projective and indecomposable injective module, respectively, corresponding to the primitive idempotent $e_i$. \newline
The dominant dimension domdim($M$) of a module $M$ with a minimal injective resolution $(I_i): 0 \rightarrow M \rightarrow I_0 \rightarrow I_1 \rightarrow ...$ is defined as: \newline
domdim($M$):=$\sup \{ n | I_i $ is projective for $i=0,1,...,n \}$+1, if $I_0$ is projective, and \newline domdim($M$):=0, if $I_0$ is not projective. \newline
The codominant dimension of a module $M$ is defined as the dominant dimension of the $A^{op}$-module $D(M)$.
The dominant dimension of a finite dimensional algebra is defined as the dominant dimension of the regular module. It can be shown that the dominant dimension of an algebra always equals the dominant dimension of the opposite algebra, see for example \cite{Ta}.
So domdim($A$)$ \geq 1$ means that the injective hull of the regular module $A$ is projective or equivalently, that there exists an idempotent $e$ such that $eA$ is a minimal faithful projective-injective module.
Unless otherwise stated, $e$ without an index will always denote the idempotent such that $eA$ is the minimal faithful injective-projective $A$-module in case $A$ has dominant dimension at least one.
Algebras with dominant dimension larger than or equal to 1 are called QF-3 algebras.
All Nakayama algebras are QF-3 algebras (see \cite{Abr}, Proposition 4.2.2 and Propositon 4.3.3).
For more information on dominant dimensions and QF-3 algebras, we refer to \cite{Ta}.
\begin{definition}
$A$ is called a Morita algebra iff it has dominant dimension larger than or equal to 2 and $D(Ae) \cong eA$ as $A$-right modules. This is equivalent to $A$ being isomorphic to $End_B(M)$, where $B$ is a selfinjective algebra and $M$ a generator of mod-$B$ and in this case $B=eAe$ and $M=D(eA)$ (see \cite{KerYam}).
$A$ is called a gendo-symmetric algebra iff it has dominant dimension larger than or equal to 2 and $D(Ae) \cong eA$ as $(eAe,A)-$bimodules iff it has dominant dimension larger than or equal to 2 and $D(eA) \cong Ae$ as $(A,eAe)$-bimodules. This is equivalent to $A$ being isomorphic to $End_B(M)$, where $B$ is a symmetric algebra and $M$ a generator of mod-$B$ and in this case $B=eAe$ and $M=Ae$ (see \cite{FanKoe}).
\end{definition}
We assume the reader to be familiar with Nakayama algebras. See chapter 32 in \cite{AnFul} or chapter 5 in \cite{ASS} for more on this topic and \cite{Mar} for the calculation of minimal projective resolutions and minimal injective coresolutions for Nakayama algebras. We just give here some conventions.
Let $A$ be a finite dimensional connected Nakayama algebra given by quiver and relations for the rest of this paragraph. Thus their quiver is a directed line or a directed circle. We choose to number the points in the quiver by $0,1,...,n-1$ in a clockwise way in case the algebra has $n$ simple modules.
In this case, the algebra is uniquely determined by the sequence $c=(c_0,c_1,...,c_{n-1})$ (see \cite{AnFul}, Theorem 32.9.), where $c_i$ denotes the dimension of the indecomposable projective module $e_iA$ and $n$ is the number of simple modules. The sequence $(c_0,c_1,...,c_{n-1})$ is called the Kupisch series of $A$. We look at the indices of the $c_i$ always modulo $n$. Thus $c_i$ is defined for every $i \in \mathbb{Z}$.
Every indecomposable module over a Nakayama algebra is isomorphic to $e_iA/e_iJ^k$ for some $k \in \{1,2,...,c_i\}$. For $k=c_i$, one gets exactly the indecomposable projective modules.
Recall the following definitions from \cite{Mar}:
\begin{definition}
For a finite dimensional algebra $A$ and a module $M$ we define $\phi_M$ as \newline $\phi_M:= \inf \{ r \geq 1 | Ext_{A}^{r}(M,M)\neq 0 \}$ with the convention $\inf(\emptyset)= \infty$.
We also define $\Delta_A:= \sup \{ \phi_M | M$ is a nonprojective generator-cogenerator $\}$.
Clearly, for a nonselfinjective algebra $A$, it holds that $\Delta_A=  \inf \{ r \geq 1 | Ext_{A}^{r}(D(A),A) \neq 0 \}$ and for a selfinjective algebra $A$ we get that $\Delta_A=  \sup \{ \phi_M | M$ is a non-projective A-module$ \}$.
\end{definition}
One reason of interest in this Definition is the following theorem of Mueller and the connection to dominant dimension:
\begin{theorem}
(see \cite{Yam} or \cite{Mue})
Let $B=End_A(M)$ for a generator-cogenerator $M$ of $mod-A$, then $domdim(B)=\phi_M+1$.
\end{theorem}
By the previous theorem, it is equivalent to give upper bounds for $\Delta_A$ and upper bounds on the dominant dimensions of algebras isomorphic to $End_A(M)$, where $M$ is a generator-cogenerator of $mod-A$.
We also assume the reader to be familiar with the theory of finite dimensional Hopf algebras, which is explained in detail for example in \cite{SkoYam}, chapter VI.

\section{Upper bounds for the dominant dimension of gendo-symmetric algebras}
\subsection{General upper bounds}
In this chapter we give bounds for the dominant dimension of a gendo-symmetric algebra $A$ over a field $K$ depending on other invariants of the algebra, which we now introduce.
For an $A$-module $M$ we call the number of nonisomorphic indecomposable direct summands the size of $M$ and denote it by size($M$). For a subcategory $C$ of mod-$A$ we say $C \perp_r C$ iff $Ext^{k}(X,Y)=0$, for all $X,Y \in C$ and $k=1,2,...,r$.
For a module $M$, B($M$) denotes the basic version of the module $M$.
\begin{definition}
\label{blabla}
We define $o_k(A):=sup \{ $ size(M) $|$ $M \in mod-A$, add(M) $\perp_k$ add(M) $ \}$.
Note that $o_k(A)$ is always larger than or equal to the number of simple A-modules, since the size of the regular module is equal to the number of simple A-modules.
A module $M$ is called $k$-rigid, if add($M$) $\perp_k$ add($M$). 
\end{definition}

For the next lemma, see for example \cite{Iya2} section 2.1.
\begin{lemma}
\label{lemma von iyama}
The following holds in the stable category for a module $X$ with $X \perp_n add(A)$ and an arbitrary module $Y \in mod-A$ for $n \geq 1$:
\begin{enumerate}
\item[a)]$\Omega^n: \underline{Hom}(X,Y) \rightarrow \underline{Hom}(\Omega^{n}(X),\Omega^n{(Y)})$ is a k-isomorphism.
\item[b)]We have a functorial isomorphism $\underline{Hom}(\Omega^{n}(X),Y)=Ext^{n}(X,Y)$.
\end{enumerate}
\end{lemma}

\begin{theorem}
\label{theoremiyama}
\begin{enumerate}
\item[a)]If mod-A has a maximal 1-orthogonal module M, then $o_1(A)$ equals size(M).
\item[b)]If the representation dimension of A is smaller than or equal to 3, then $o_1(A)$ is finite.
\end{enumerate}
\end{theorem}
\begin{proof}
see \cite{Iya2} 5.5.1.
\end{proof}

\begin{theorem}
\label{lemma domdim}
Let $A$ be a gendo-symmetric algebra.
For an $A$-module $M$, $domdim(M)=n$ $\geq 2$ iff $Hom_A(D(A),M) \cong M$ and $Ext^{i}(D(A),M) =0$ for 
$i=1,..,n-2$ and $Ext^{n-1}(D(A),M) \neq 0$.
\end{theorem}
\begin{proof}
see \cite{FanKoe2} Proposition 3.3.
\end{proof}

\begin{lemma}
\label{lemma1}
Let $A$ be a non-selfinjective gendo-symmetric algebra.
Assume $domdim(A)=n+2 \geq 2$ and $k \geq 1$. \newline
The module M:=$\bigoplus\limits_{l=0}^{q}{\Omega^{i_l}(D(A))}$ satisfies add($M$) $\perp_k$ add($M$), in case the sequence $i_l$ satisfies the following conditions:
\begin{enumerate}
\item[i)]$i_{l+1}-i_l \geq k+2$ for $l=0,...,q-1$,
\item[ii)]$k+i_q \leq n$.
\end{enumerate}
\end{lemma}

\begin{proof}
Since the dominant dimension is left-right symmetric, we have that domdim($\Omega^{t}(D(A))$)=t, for all $t=1,..., n+2$, because of codomdim($D(A)$)=$n$+2.
For $i \geq j \geq 0$ and $k+i \leq n$ we have using \ref{lemma von iyama} and $add(\Omega^{i}(D(A))) \perp_k add(A)$: \newline
$Ext^k(\Omega^{i}(D(A)),\Omega^{j}(D(A))) \cong \underline{Hom}(\Omega^{k+i}(D(A)),\Omega^{j}(D(A))) \cong \underline{Hom}(\Omega^{k+i-j}(D(A)),D(A)) \cong Ext^{k+i-j}(D(A),D(A))=0$. \newline
For $i > j$ and $i+k \leq n$, we have with $s:=i-j$:
$Ext^{k}(\Omega^{j}(D(A)),\Omega^{i}(D(A))) \cong \underline{Hom}(\Omega^{k+j}(D(A)),\Omega^{i}(D(A))) \cong \underline{Hom}(\Omega^{k}(D(A)),\Omega^{s}(D(A))) \cong Ext^{k}(D(A),\Omega^{s}(D(A))=0 $ since $domdim(\Omega^{s}(D(A))) =s \geq k+2$ by {\ref{lemma domdim}}.
Thus we see that choosing the sequence $i_l$ and M as described in the lemma we get add(M) $\perp_k$ add(M).
\end{proof}

\begin{example}
We refer to \cite{Mar} or \cite{AnFul} for basics on Nakayama algebras, which will be used in this example.
Let $n \geq 5$ and $A$ be the Nakayama algebra with $n$ simple modules and Kupisch series $[n,n+1,n+1,...n+1]$. It easy to verify that $A$ is a gendo-symmetric algebra of dominant dimension $2n-2$. Define the natural number $d$ by $2n-2=d+2$ or equivalently $d=2n-4$ and let $k:=2$. Then the module $M:=\bigoplus\limits_{l=0}^{q}{\Omega^{i_l}(D(A))}$ is $2$-rigid, when choosing the sequence $i_l$ with $i_{l+1}-i_l \geq 4$ and $2+i_q \leq 2n-4$. For $n=5$, we can choose $q=1$, $i_0=0$ and $i_1=4$ and obtain the module $M=D(A) \oplus \Omega^{4}(D(A))$, which indeed is 2-rigid as can be verified by direct calcuations.

\end{example}

\begin{theorem}\label{maintheorem}
Let $A$ be a non-selfinjective gendo-symmetric algebra with dominant dimension $n+2$ and let $w$ denote the number of simple $A$-modules.
Then $(o_k(A)+2-w)(k+2)-1 \geq domdim(A)$, for all $k \geq 1$.
\end{theorem}
\begin{proof}
We define $M:= \bigoplus\limits_{l=0}^{q}{\Omega^{(k+2)l}(D(A))}$, where we choose q as the maximal natural number with (k+2)q+k $\leq $n.
Note that there is no $t$ with $1 \leq t+1 \leq n$ and $B(\bigoplus\limits_{l=0}^{t}{\Omega^{(k+2)l}(D(A))}) \cong B(\bigoplus\limits_{l=0}^{t+1}{\Omega^{(k+2)l}(D(A))})$, since we have that the injective dimension of $B(\bigoplus\limits_{l=0}^{t+1}{\Omega^{(k+2)l}(D(A))})$ is $(k+2)(t+1)$ and the injective dimension of $B(\bigoplus\limits_{l=0}^{t}{\Omega^{(k+2)l}(D(A))})$ is $(k+2)t$.
This means that $\bigoplus\limits_{l=0}^{t_1}{\Omega^{(k+2)l}(D(A))}$ has at least one indecomposable module W as a direct summand such that $ \bigoplus\limits_{l=0}^{t_2}{\Omega^{(k+2)l}(D(A))}$ does not have W as a direct summand (for $1 \leq t_1,t_2 \leq n$ and $t_1 > t_2$). Therefore, $M$ has at least size $w+q$, when $w$ denotes the number of simple modules of A ($w$ comes from the fact hat $D(A)$ has size $w$ and every new summand $\Omega^{(k+2)l}(D(A))$ of $M$ adds at least one to the size of $M$). \newline
We see that $M$ satisfies the conditions from lemma {\ref{lemma1}}, which we can now apply. \newline
By the choice of $q$, we have $(k+2)(q+1)+k > n$ or equivalently $q > \frac{n-k}{k+2}-1$ and this gives us: \newline
$o_k(A) \geq size(M) \geq w+q > w+ \frac{n-k}{k+2}-1$.
Solving this for $n$ gives: \newline
$(o_k(A)+1-w)(k+2)+k>n$ or $(o_k(A)+1-w)(k+2)+1+k=(o_k(A)+2-w)(k+2)-1 \geq n+2=domdim(A)$.
\end{proof}

Recall that the finitistic dimension conjecture implies the Nakayama conjecture (see for example \cite{Yam} for a proof) and that the finitistic dimension conjecture is true for algebras with representation dimension at most 3 (see \cite{IgTo}). Especially: The Nakayama conjecture is true for algebras with representation dimension at most 3. The next corollary is a generalisation of this fact for gendo-symmetric algebras and also gives a concrete bound depending on the maximal size of a 1-rigid module in the algebra.
\begin{corollary}
The Nakayama conjecture holds for gendo-symmetric algebras, where there exists a $k\geq 1$ with $o_k(A) < \infty$. Especially the Nakayama conjecture holds for gendo-symmetric algebras which have representation dimension at most 3 or a finitely generated maximal 1-orthogonal module and in this case we have: \newline
$3(o_1(A)+2-w)-1 \geq domdim(A)$.
\end{corollary}
\begin{proof}
The first part is an immediate corollary of \ref{maintheorem}.
The second part follows from the finiteness of $o_1(A)$ by lemma {\ref{theoremiyama}}, in case the representation dimension is at most 3 or there is a finitely generated maximal 1-orthogonal module.
\end{proof}
We note that Iyama asked in \cite{Iya2}, whether $o_1(A)$ is always finite. A negative answer was found in \cite{HIO}. 
But the question, whether there exists a $k \geq 1$ with $o_k(A)$ finite for any finite dimensional algebra $A$ seems to be still open and would imply the Nakayama conjecture for gendo-symmetric algebras by the previous corollary. \newline

Recall a conjecture of Yamagata in \cite{Yam}, who conjectures there that the dominant dimension of a nonselfinjective algebra is bounded by a function depending on the number $w$ of simple modules.
One way to prove this for a class of gendo-symmetric algebras might be to use the above theorem \ref{maintheorem} and find a $k$ (depending only on $w$) such that the number of $k$-rigid modules is finite and also depends only on $w$.
To show that this might be reasonable, we prove this for $k=1$ for Nakayama algebras (which are not necessarily gendo-symmetric). We assume for simplicity that the algebras are given by quiver and relations and have $n \geq 2$ simple modules. The calculations in the general case are the same, but one is missing the graph-theoretic interpretation then. Note also that Nakayama algebras with $n=1$ simple modules have no non-projective 1-rigid modules by the first example in \ref{examples} and thus $o_1(A)=1$ for such algebras $A$.
\begin{example}
Every Nakayama algebra with a line as a quiver with $n$ simple modules is a quotient of a hereditary representation-finite algebra of type $A_n$, so the number of indecomposable modules is bounded by $n(n+1)$ and therefore also the number of indecomposable 1-rigid modules is bounded by $n(n+1)$.
So assume now that we have a nonselfinjective Nakayama algebra with a circle as a quiver and $n$ simple modules and the projective indecomposables at the point $i$ have length $c_i$. The points in the quiver are numbered from 0 to $n-1$ in a clockwise manner. Note that in this section we do not assume that $c_{n-1}=c_0+1$, as some authors always do for Kupisch series of Nakayama algebras.
Without loss of generality we can assume that we have a nonprojective indecomposable module $M$ of the form $M \cong e_0 A/e_0 J^{k}$. To calculate $Ext^{1}(M,M)$ we look at the minimal projective presentation of $M$: \newline $e_{c_0}A \stackrel{L_{k,c_0-k}}{\rightarrow} e_{k}A \stackrel{L_{0,k}}{\rightarrow} e_0 A \rightarrow M \rightarrow 0$. Here a homomorphism of the form $L_{x,y}$ denotes the left multiplication with the path $w_{x,y}$, which starts at $x$ and has length $y$.
Applying the functor $Hom(-,M)$ to $e_{c_0}A \stackrel{L_{k,c_0-k}}{\rightarrow} e_{k}A \stackrel{L_{0,k}}{\rightarrow} e_0 A$ we get: 
$0 \rightarrow (e_0 A /e_0 J^{k})e_0 \stackrel{R_{0,k}}{\rightarrow} (e_0 A /e_0 J^{k})e_k \stackrel{R_{k,c_0-k}}{\rightarrow} (e_0 A/e_0 J^{k})e_{c_0}$.
Here $R_{x,y}$ denotes right multiplication by $w_{x,y}$.
Note that we have $R_{0,k}$=0, since right multiplication by a path with length $k$ vanishes since it maps to $(e_0 A /e_0 J^{k})e_k$. Thus we have $Ext^{1}(M,M) \cong ker(R_{k,c_0-k}).$
Now we want to find a condition on $k$ characterising $Ext^{1}(M,M)  \neq 0$.
We write $k=q+sn$, for a $q$ with $0 \leq q \leq n-1$ and a $s \geq 1$ (we will consider the case $s=0$ later). \newline
Then $(e_0 A /e_0 J^{k})e_k = \langle  w_{0,q} w_{k,n}^l \mid ln+q<k  \rangle$ (here we denote by $\langle \cdots \rangle$ the vectorspace span) and 
$ ker(R_{k,c_0-k}) = \langle w_{0,q} w_{k,n}^l \mid ln+q < k $ and $ ln+q+c_0-k \geq k \rangle$.
So the longest path in $(e_0 A /e_0 J^{k})e_k = \langle w_{0,q} w_{k,n}^l \mid ln+q < k \rangle$ has length equal to $n(s-1)+q$.
Now we have $ker(R_{k,c_0-k}) \neq 0$ iff $ker(R_{k,c_0-k})$ contains the path of length $(s-1)n+q$, since if there is some element in the kernel then the path with longest length must also be in the kernel.
This gives us $Ext^{1}(M,M)=ker(R_{k,c_0-k}) \neq 0$ iff $(s-1)n+q+c_0-k \geq k$ iff $c_0 \geq n+k$ iff $c_0-n \geq k$. Thus $Ext^{1}(M,M)=0$ iff $c_0-n < k$.
Now we look at the case $k=q$, for a $q$ with $0<q \leq n-1$.
Then $(e_0 A /e_0 J^{q})e_q= 0$ and thus in this case we always have $Ext^{1}(M,M)=0$. 
We now collect our findings: 
\begin{proposition}
Let $A$ be a Nakayama algebra with $n \geq 2$ simple modules, then an indecomposable module of the form $e_iA/e_iJ^k$ is 1-rigid iff $1 \leq k \leq n-1$ or $k>c_i-n$. The number of indecomposable 1-rigid modules is therefore bounded by $n(n-1)+n^2$. Especially $o_1(A) \leq n(n-1)+n^2$ is bounded by a function depending on the simple modules $n$.
\end{proposition}
In \cite{Mar}, it was shown that the optimal bound for the dominant dimension of (gendo-symmetric) Nakayama algebras with $n$ simple modules is $2n-2$. 
\end{example}

\subsection{Upper bounds for 1-Extsymmetric algebras}
In this short subsection we introduce 1-Extsymmetric algebras as a generalisation of weakly 2-Calabi-Yau algebras and give explicit bounds for the dominant dimensions of the corresponding Morita algebras.
\begin{definition}
We call a selfinjective finite dimensional algebra 1-Extsymmetric in case the following holds for any modules $X,Y$:
$Ext^{1}(X,Y) \neq 0$ iff $Ext^{1}(Y,X) \neq 0$.
\end{definition}
\begin{example}
Recall from \cite{CheKoe} that a selfinjective algebra $A$ is called weakly 2-Calabi-Yau in case its stable category is weakly 2-Calabi-Yau as a triangulated category. For a $K$-linear triangulated category $T$, this means that there is a natural
isomorphism $Hom_T(Y,X[2]) \cong DHom_T(X,Y)$ for any $X,Y \in T$.  
For weakly 2-Calabi-Yau algebras there there is an isomorphism $Ext^{1}(X,Y) \cong Ext^{1}(Y,X)$ as $K$-vector spaces for any modules $X,Y$, this is noted for example in \cite{CheKoe}, lemma 5.2. Thus weakly 2-Calabi-Yau algebras are 1-Extsymmetric.
Famous examples of weakly 2-Calabi-Yau algebras are preprojective algebras of Dynkin type.
\end{example}

\begin{theorem}
Let $A$ be a selfinjective 1-Extsymmetric algebra with $s$ simple modules. Then $\Delta_A \leq o_1(A)+s-2$.
Let $B=End_A(M)$ for a generator $M$ of $mod-A$, then $domdim(B) \leq \Delta_A+1 \leq o_1(A)+s-1$.
\end{theorem}
\begin{proof}
Let $M$ be a module with $Ext^{i}(M,M)=0$ for all $i=1,2,...,n$. Then $Ext^{1}(M,\Omega^{j}(M))=0$ because of $Ext^{1}(\Omega^{j}(M),M)=0$ for all $j=0,1,...,n-1$, using that $Ext^{i}(M,M) \cong Ext^{1}(\Omega^{i-1}(M),M)$. Then one also has $Ext^{1}(\Omega^{p}(M),\Omega^{q}(M))=0$ for all $p,q \in \{0,1,...,n-1 \}$ because in case $p \geq q$, one has (using that $\Omega^{1}$ is an equivalence in the stable category) $Ext^{1}(\Omega^{p}(M), \Omega^{q}(M)) \cong Ext^{1}(\Omega^{p-q}(M),M) =0$ and similar for $p<q$ one has $Ext^{1}(\Omega^{p}(M), \Omega^{q}(M)) \cong Ext^{1}(M,\Omega^{q-p}(M))=0$. \newline Therefore $Ext^{1}(A \oplus \bigoplus\limits_{l=0}^{n-1}{\Omega^{l}(M)} \oplus \bigoplus\limits_{l=0}^{n-1}{\Omega^{l}(M)})=0$ and thus $o_1(A) \geq n+s$. The statement about $B$ is an immediate consequence of Mueller's theorem.
\end{proof}

\begin{corollary}
The Tachikawa conjecture is true for selfinjective 1-Extsymmetric algebras $A$, in case $o_1(A)$ is finite. This is the case for example if the representation dimension of $A$ is smaller than or equal 3.
\end{corollary}

\section{Non-vanishing of certain Ext-groups}

In this chapter let an algebra always be a finite dimensional nonsemisimple selfinjective k-algebra, over a field $K$. To avoid trivialities, we futhermore assume that our algebras are not isomorphic to $K[x]/(x^2)$. Note that any algebra derived equivalent to $K[x]/(x^2)$ is in fact Morita equivalent to $K[x]/(x^2)$.
\begin{lemma}
Let $A$ and $B$ be selfinjective algebras.
\begin{enumerate}
\item[i)] If $F: \underline{mod-A} \rightarrow \underline{mod-B}$ is a stable equivalence, then $F\Omega_A = \Omega_B F$.
\item[ii)]$Ext^{i}(M,N) \cong \underline{Hom_A}(\Omega^{i}(M),N)$, for every $i \geq 1$ and modules $M,N$.
\end{enumerate}
\end{lemma}
\begin{proof}
\begin{enumerate}
\item See \cite{ARS}, chapter X proposition 1.12.
\item See \cite{SkoYam}, chapter IV. theorem 9.6.

\end{enumerate}
\end{proof}

\begin{theorem}
If two selfinjective algebras $A$ and $B$ are derived equivalent or stable equivalent, then $\Delta_A = \Delta_B$.
\begin{proof}
Since a derived equivalence between selfinjective algebras induces a stable equivalence, we just have to show that we have $\Delta_A = \Delta_B$ in case $A$ and $B$ are stable equivalent. \newline
Let $M$ be a nonprojective $A-module$ and $F:\underline{mod-A} \rightarrow \underline{mod-B}$ be a stable equivalence.
Using the previous lemma, one obtains:
$$Ext_{B}^{r}(F(M),F(M)) \cong \underline{Hom}_B(\Omega_B^{r}(F(M)),F(M)) \cong  F(\underline{Hom}_A(\Omega_A^{r}(M),M)) \neq 0.$$
Thus we have $\Delta_B \geq \Delta_A$. By symmetry we obtain: $\Delta_A = \Delta_B$.
\end{proof}
\end{theorem}
\begin{theorem}
Assume $A$ is symmetric.
For a nontrivial twosided ideal $X$ of $A$ with $Hom_A(X,A/X) \neq 0$ the following holds: $\phi_X=1$.
\end{theorem}
\begin{proof}
We have a short exact sequence of $A$-bimodules:
$$0 \rightarrow X \rightarrow A \rightarrow A/X \rightarrow 0 (*)$$
and applying $Hom_A(X,-)$, we get the exact sequence: 
$$0 \rightarrow Hom_A(X,X) \rightarrow Hom_A(X,A) \rightarrow Hom_A(X,A/X) \rightarrow Ext_A^{1}(X,X) \rightarrow 0.$$
From this exact sequence we conclude that we have a short exact sequence
$$0 \rightarrow Hom_A(X,X) \rightarrow Hom_A(X,A) \rightarrow Hom_A(X,A/X) \rightarrow 0$$
iff $Ext_A^{1}(X,X)=0$.
The last sequence being short exact is equivalent to:
$$dim(X)=dim(Hom_A(X,A))=dim(Hom_A(X,X))+dim(Hom_A(X,A/X)).(**)$$
Dualising $(*)$ gives the short exact sequence of $A$-bimodules:
$$0 \rightarrow D(A/X) \rightarrow D(A) \rightarrow D(X) \rightarrow 0.$$
Using that $D(X)$ is a bimodule and $D(A) \cong A$ as $A$-bimodules, we conclude that $D(A/X) \cong I$ as $A$-bidmodules, for a twosided ideal $I$. We then have $Hom_A(X,X) \cong Hom_A(D(X),D(X)) \cong Hom_A(A/I,A/I) \cong A/I$ and from this we have $dim(Hom_A(X,X))=dim(A/I)=dim(A)-dim(I)=dim(A)-dim(D(A/X))=dim(A)-(dim(A)-dim(X))=dim(X).$
Since we have $dim(Hom_A(X,A/X)) \neq 0$ we conclude that $(**)$ can not hold and thus we have $Ext_A^{1}(X,X) \neq 0$.
\end{proof}
\begin{corollary}\label{idealcor}
For a nontrivial twosided ideal $X \neq A$ of a local symmetric k-algebra A we have $Hom_A(X,A/X) \neq 0$ and thus $\phi_X=1$.
\end{corollary}
\begin{proof}
Since there is a unique simple module, there is a map which maps from $X$ to the simple module which is a part of the socle of $A/X$. Thus $Hom_A(X,A/X) \neq 0$ and we can apply the above theorem.
\end{proof}

\begin{corollary}\label{commutativecase}
Let $A$ be a local symmetric algebra that is also commutative with enveloping algebra $A^{e}$. Then $Ext_{A^{e}}^{1}(A,A) \neq 0$.
\end{corollary}
\begin{proof}
We show that $A$ is a twosided ideal of $A^{e}$ and then we can apply \ref{idealcor} to show the result. But the socle of $A$ is a two-sided ideal of $A$ and is also simple, since $A$ is selfinjective and local. We have that $soc_{A^{e}}(A)$ is a subbimodule of $soc(A)$ in general, since the sum of simple subbimodules of $A$ is a submodule of the sum of simple right modules. But since  $soc_{A^{e}}(A)$ is non-trivial, one obtains $soc_{A^{e}}(A)=soc(A)$, which is simple. Thus $A$ is a two sided ideal in the commutative algebra $A^{e}$.  

\end{proof}
\begin{example} \label{examples}
We give some examples, showing that the previous corollary \ref{idealcor} can be applied to show that certain algebras coming from applied representation theory have dominant dimension exactly two.
\begin{enumerate} 
\item[i)]The indecomposable modules of a local Nakayama algebra $k[X]/(X^n)$ are all twosided ideal of the form $(X^k)/(X^n)$ and we can thus apply \ref{idealcor} to get $\Delta_{k[X]/(X^n)}=1$. 
Thus all algebras of the form $End_{k[X]/(X^n)}(M)$ have dominant dimension 2 in case $M$ is a non-projective generator.
\item[ii)] Category $\mathcal{O}$ blocks have dominant dimension 2, since they are isomorphic to endomorphism rings of generators over a local symmetric commutative algebra having a two-sided ideal as a direct summand, see for example \cite{KSX} and \cite{Hum} for more information on those algebras.
\item[iii)] 1-quasi-hereditary algebras with a special duality in the sense of \cite{Pu} Theorem B have dominant dimension equal to two for the same reason as blocks of category $\mathcal{O}$.
\end{enumerate}
\end{example}

\begin{lemma}\label{hopflemmas}
For a finite dimensional Hopf algebra $A$ and $A$-modules $M_1$, $M_2$ and $M_3$, then the following holds:
\begin{enumerate}
\item[i)]$Ext_A^{i}(M_1 \otimes_k M_2 , M_3) \cong Ext_A^{i}(M_1,Hom_k(M_2,M_3))$, for every $i \geq 1$.
\item[ii)]$Hom_A(M_1,M_2) \cong M_1^{*} \otimes_k M_2$
\item[iii)] $M_1$ is projective iff $M_1 \otimes_k M_1^{*}$ is projective.
\item[iv)]Let now $A=kG$ be group algebra over an algebraic closed field of characteristic p. $M_1^{*} \otimes_k M_1$ has the trivial module $K$ as a direct summand iff $p$ does not divide $dim(M_1)$.
\item[v)]$A$ is selfinjective.
\end{enumerate}
\end{lemma}
\begin{proof}
\begin{enumerate}
\item See \cite{SkoYam}, theorem 6.4. for $i$=0 and for $i>0$ the proof is as in proposition 3.1.8. (ii) of \cite{Ben}.
\item See \cite{SkoYam}, chapter VI. exercise 24.
\item See \cite{SkoYam}, chapter VI. exercise 27.
\item See \cite{Ben}, theorem 3.1.9.
\item See \cite{SkoYam}, theorem 3.6.
\end{enumerate}
\end{proof}
The following is a generalisation of a result of Tachikawa from \cite{Ta} (who proved it for group algebras of $p$-groups). The proof is similar but shorter, since we use the previous lemma.
\begin{theorem} \label{tachtheo}
For a local nonsemisimple finite dimensional Hopf algebra $A$ we have $\Delta_A=1$.
\end{theorem}
\begin{proof}
We have to show for every nonprojective module $M$, that we have $Ext_A^{1}(M,M) \neq 0$.
Using i) and ii) of the above lemma we have:
$Ext_A^{1}(M,M) \cong Ext_A^{1}(k \otimes_k M , M) \cong Ext_A^{1}(k,Hom_k(M,M)) \cong Ext_A^{1}(k,M \otimes_k M^{*})$.
Now using iii) of the above lemma, we see that $M \otimes_k M^{*}$ is not projective since $M$ is not projective. But since $A$ is local and selfinjective, the first term of a minimal injective coresolution of $M \otimes_k M^{*}$ has $A$ as a direct summand and thus $Ext_A^{1}(k,M \otimes_k M^{*}) \neq 0$.
\end{proof}
\begin{corollary}
Let $B=End_A(M)$, where $A$ is a local Hopf algebra and $M$ a non-projective generator of $mod-A$, then $B$ has dominant dimension equal to two.
\end{corollary}
\begin{example}\label{hopfexample}
To show that the previous theorem is really a generalisation of the result of Tachikawa, one has to find a finite dimensional local Hopf algebra that is not isomorphic to a group algebra.
Xingting Wang suggested to try example (A5) from theorem 1.1 in the paper \cite{NWW}. Here we give the proof that it is not isomorphic to a group algebra by calculating the quiver with relations isomorphic to the algebra and then calculating the beginning of a minimal projective resolution of the simple module. Fix an algebraically closed field $K$ of characteristic 2. The algebra $A$ is defined as $K<x,y,z>/(x^2,y^2,xy-yx,xz-zy,yz-zy-x,z^2-xy)$ (for the Hopf algebra structure see the above cited paper). We will describe this example by quiver and relations.
Note first that $A$ is local of dimension 8 over the field $K$ with basis $\{1,x,y,z,z^2,xz,yz,zy\}$ and the Jacobsonradical is the ideal generated by $x,y$ and $z$. Now we have to calculate the second power $J^2$ of the Jacobsonradical: It contains $x$, since $x=yz-zy \in J^2$. Since $A$ is not commutative, its quiver can not have just one loop. Thus the dimension of $J^2$ is at most 5. It is clear that $J^2$ contains every basis element expect possibly $y$ and $z$. Thus, since the dimension of $J^2$ is at most 5, $J^2$ has basis $x,z^2,xz,yz,zy$. We will now show that the quiver algebra of $A$ is equal to $k<a,b>/(a^2,b^2-aba)$. Clearly $k<a,b>$ maps onto $A$ by a map $f$, with $f(a)=y$ and $f(b)=z$. Note that $(a^2,b^2-aba)$ is contained in the Kernel of $f$, since $y^2=0$ and $z^2-yzy=z^2-(zy+x)y=z^2-xy=0$. Thus there is a surjective map $\hat{f}:k<a,b>/(a^2,b^2-aba) \rightarrow A$ induced by $f$. But since $k<a,b>/(a^2,b^2-aba)$ also has dimension 8, that is in fact an isomorphism. \newline
Now we show that $A=k<a,b>/(a^2,b^2-aba)$(we will identify $A$ now with $k<a,b>/(a^2,b^2-aba)$)  is not isomorphic to a group algebra. Since $A$ has dimension 8 and is not commutative, there are only 2 candidates of group algebras, that could be isomorphic to $A$: The group algebra of the dihedral group of order 8 and the group algebra of the Quaternion group.
We will now calculate the beginning of a projective resolution of the simple module $S_1$ of $A$. Note that $A$ has basis $\{1,a,b,ab,ba,aba=b^2,bab,baba=b^3=abab \}$ and thus Loewy length $4$. It is clear that $\Omega^{1}(S_1)$ equals the radical $J$ of the algebra. Since $J/J^2=<\overline{a},\overline{b}>$, the projective cover $f_1 : A^2 \rightarrow J$ is given by $f_1(x_1,x_2)=ax_1+bx_2$.
Then $\Omega^{2}(S_1)=ker(f_1)= \newline \langle (a,0),(ab,0),(aba,0),(baba,0),(0,ba),(0,bab),(0,baba),(bab,aba),(ba,b) \rangle $ is of dimension 9. \newline Top($\Omega^{2}(S_2))= \langle \overline{(a,0)},\overline{(ba,b)} \rangle$ and thus the projective Cover of $\Omega^{2}(S_1)$ is of the form $f_2 : A^2 \rightarrow \Omega^{2}(S_1)$ with $f_2(x_1,x_2)=(a,0)x_1+(ba,b)x_2$. \newline Then $\Omega^{3}(S_1)= \langle (a,0),(ab,0),(aba,0),(abab,0),(0,b^3),(0,bab),(bab-ba) \rangle$ and \newline Top($\Omega^{3}(S_1))= \langle \overline{(a,0)},\overline{(bab,-ba)}) \rangle$.
Thus the projective Cover of $\Omega^{3}(S_1)$ must have the form $f_3: A^2 \rightarrow \Omega^{3}(S_1)$ and thus $\Omega^{4}(S_1)$ has dimension 9.
The dimension of $\Omega^{4}(S_1)$ is the crucial information that we need to distinguish $A$ from group algebras of dimension 8 over the field. Let $B$ be the group algebra of the dihedral of order 8 over $K$ with simple module $S_2$. Then by \cite{Ben2} chapter 5.13., $\Omega^{4}(S_2)$ has dimension 17 and thus $A$ is not isomorphic to $B$. Let $C$ be the group algebra of the quaternion group of order 8 over $K$ with simple module $S_3$. Then this algebra is 4-periodic (see for example \cite{Erd}) and thus $\Omega^{4}(S_3) \cong S_3$ and $A$ is not isomorphic to $C$. This shows that $A$ is not isomorphic to any group algebra.

\end{example}
\begin{theorem}
For a group algebra $kG$ over an algebraic closed field of characteristic $p$ we have: \newline
$\phi_M \leq \phi_k$, for every module $M$ such that p does not divide $dim(M)$.
\end{theorem}
\begin{proof}
Since p does not divide $dim(M)$, we can use iv) of \ref{hopflemmas} and we can write $M \otimes_k M^{*}=k \oplus_A N$, with another module $N$.
Using \ref{hopflemmas}, one obtains: $Ext_A^{i}(M,M) \cong Ext_A^{i}(M \otimes_k k , M) \cong Ext_A^{i}(k,Hom_k(M,M)) \newline \cong Ext_A^{i}(k,M \otimes_k M^{*})=Ext_A^{i}(k,k)\oplus Ext_A^{}(k,N)$
Thus if we have $Ext^{r}(k,k) \neq 0$ for an $r \geq 1$, then we also have $Ext_A^{r}(M,M) \neq 0$.
\end{proof}
\begin{theorem}
For a group algebra $kG$ of finite representation type and a block $B$ we have: \newline
$\Delta_{B}=\phi_S=2s-1$, where s is the number of simple modules in $B$ and S is an arbitrary simple module with uniserial projective cover of $B$.
\end{theorem}
\begin{proof}
In \cite{Mar} we proved the result in case $B$ is a symmetric Nakayama algebra. But in general such $B$ is stable equivalent to a symmetric Nakayama algebra and so we can use that $\Delta_B$ is invariant under stable equivalence to get the result, since by \cite{ChaMar} section 2.4. there is a stable equivalence between a general representation-finite block and a symmetric Nakayama algebra sending a simple module with uniserial projective cover to an arbitrary simple module of the symmetric Nakayama algebra.
\end{proof}
\section{Open questions and comments}
This article motivated the following questions:
For the first two questions let $A=kG$ be a group algebra and $B$ a block of $A$, but note that those questions might also be interesting for a general selfinjective $B$.
\begin{enumerate}
\item[i)] Is it true that $\Delta_B = max \{\phi_S \mid S $ a simple $B-$module$ \}$? This hold for p-groups (since they are local Hopf algebras) and for group algebras of finite representation type as we have shown in this article. We remark that in forthcoming work we show that this also holds for any symmetric representation-finite algebra. 
\item[ii)] Is it true that $max \{\phi_S \mid S$ a simple $B$-module $\}\leq 2s-1$, where $s$ is the number of simple modules of $B$? Again this is true for p-groups and blocks such that $B$ has finite representation type. In forthcoming work we show that this also holds for any symmetric representation-finite algebra. We also tested this for several other blocks of group algebras of the symmetric and alternating groups.
\item[iii)] Does $Ext^{1}(M,M) \neq 0$ hold for all nonprojective modules $M$ of a symmetric (or even just selfinjective) local algebra? In \cite{Hos} theorem 3.4., it was shown that the answer is yes in case the selfinjective local algebra has radical cubed zero and in this article we showed it for local Hopf algebras. Of course one can ask this question in a more general way: Characterise the selfinjective local algebras with the property that $Ext^{1}(M,M) \neq 0$ for any non-projective module $M$. We mention the following special case of this question: Given a local selfinjective algebra and a module $M$ with simple socle. Is $Ext^{1}(M,M) \neq 0$? This would imply for example that all 1-quasi hereditary algebras have dominant dimension exactly two.             
\item[iv)] Let $A$ be a gendo-symmetric algebra with a maximal 1-orthogonal object or having representation dimension at most 3. Is there a bound on $o_1(A)$ depending on the number of simples of $A$? Note that such a bound would imply Yamagata's conjecture for this class of algebras.
\item[v)] Does for any finite dimensional (gendo-symmetric) algebra $A$ exist a $k \geq 1$ with $o_k(A) < \infty$? Note that in the gendo-symmetric case a positive answer would imply the Nakayama conjecture or equivalently the Tachikawa conjecture for symmetric algebras as we have shown in this article.
\item[vi)] Is there a local Hopf algebra that is not symmetric? The author is not an expert in Hopf algebras but asked several people, who work on that area. No one could provide an example of a local non-symmetric Hopf algebra yet. You can find this question also on mathoverflow: http://mathoverflow.net/questions/202339/examples-of-local-nonsemisimple-nonsymmetric-hopf-algebras.
\end{enumerate}
We remark that in forthcoming work we prove another special case of Yamagata's conjecture by giving explicit bounds for the dominant dimension of Morita algebras $A$ such that $eAe$ is representation-finite, where $eA$ is the minimal faithful projective-injective $A$-module.

\end{document}